\documentclass[11pt,a4paper]{article}
\usepackage{pgf}
\usepackage{tikz}
\usetikzlibrary{arrows,automata}
\usepackage{verbatim}
\usepackage{caption}
\usepackage{epsf,epsfig,amsfonts,amsgen,amsmath,amstext,amsbsy,amsopn,amsthm,amssymb}
\usepackage{color}
\usepackage{mathrsfs,hyperref}
\usepackage{enumerate}
\usepackage{url}
\usepackage{enumitem}
\usetikzlibrary{patterns}

\usepackage{setspace}

\newtheorem{dfn}{Definition}[section]
\newtheorem{thm}{Theorem}[section]
\newtheorem{thmo}{Theorem}[section]

\newtheorem{claim}[thmo]{Claim}

\newtheorem{lem}[thm]{Lemma}

\newtheorem{prop}[thm]{Proposition}

\newtheorem{corollary}[thm]{Corollary}
\newtheorem{conjecture}[thm]{Conjecture}

\def\QED{\hfill \rule{7pt}{7pt}}

\newcommand{\dG}{{\mathcal{G}}}

\makeatletter

\newcommand{\Rmnum}[1]{\expandafter\@slowromancap\romannumeral #1@}
\makeatother

\begin{document}
	\title{Counting triangles in regular graphs}
	
	\author{Jialin He$^1$\and
		Xinmin Hou$^{2,4}$\and
		Jie Ma$^{3,4}$\and
		Tianying Xie$^{5}$}	
	\date{}
	
	\maketitle
    \footnotetext[1]{Department of Mathematics, Hong Kong University of Science and Technology, Clear Water Bay, Kowloon 999077, Hong Kong. Partially supported by Hong Kong RGC grant GRF 16308219 and Hong Kong RGC grant ECS 26304920. Email: majlhe@ust.hk}
    \footnotetext[2]{CAS Key Laboratory of Wu Wen-Tsun Mathematics, University of Science and Technology of China, Hefei, Anhui 230026, PR China. Partially supported by the National Natural Science Foundation of China grant 12071453. Email: xmhou@ustc.edu.cn}
    \footnotetext[3]{School of Mathematical Sciences, University of Science and Technology of China, Hefei, Anhui, 230026, China. Research supported by National Key Research and Development Program of China 2020YFA0713100, National Natural Science Foundation of China grant 12125106, and Anhui Initiative in Quantum Information Technologies grant AHY150200. Email: jiema@ustc.edu.cn.}
    \footnotetext[4]{Hefei National Laboratory, University of Science and Technology of China, Hefei 230088, China. Research supported by Innovation Program for Quantum Science and Technology 2021ZD0302902.}
    \footnotetext[5]{School of Mathematical Sciences, University of Science and Technology of China, Hefei, Anhui, 230026, China. Email: xiety@ustc.edu.cn}

\begin{abstract}
In this paper, we investigate the minimum number of triangles, denoted by $t(n,k)$, in $n$-vertex $k$-regular graphs, where $n$ is an odd integer and $k$ is an even integer.
The well-known Andr\'asfai-Erd\H{o}s-S\'os Theorem has established that $t(n,k)>0$ if $k>\frac{2n}{5}$.
In a striking work, Lo has provided the exact value of $t(n,k)$ for sufficiently large $n$, given that $\frac{2n}{5}+\frac{12\sqrt{n}}{5}<k<\frac{n}{2}$.
Here, we bridge the gap between the aforementioned results by determining the precise value of $t(n,k)$ in the entire range $\frac{2n}{5}<k<\frac{n}{2}$.
This confirms a conjecture of Cambie, de Joannis de Verclos, and Kang for sufficiently large $n$.
\end{abstract}

\section{Introduction}
We investigate the number of triangles guaranteed in regular graphs with edge density below one-half.
To be precise, for positive integers $n$ and $k$,
we define $t(n,k)$ to be the minimum number of triangles overall $k$-regular graphs on $n$ vertices.
In the case when $n$ is even, it is easy to see that $t(n,k)=0$ if and only if $k\leq \frac{n}{2}$,
which can be derived from the classic Mantel's Theorem \cite{M07} and the existence of $n$-vertex $k$-regular bipartite graphs for any $k\leq \frac{n}{2}$.
In this paper, we mainly focus on the case when $n$ is odd.\footnote{If $n$ is odd, it only makes sense to consider when $k$ is even. This is because there exist no regular graphs with an odd number of vertices and odd degrees.}
A special case of a celebrated theorem of Andr\'asfai, Erd\H{o}s and S\'os \cite{AES} states that
every $n$-vertex graph with minimum degree greater than $\frac{2n}{5}$ either is bipartite or contains some triangles.
Since there exist no bipartite regular graphs with an odd number of vertices,
this result implies that in the case when $n$ is odd, we have $t(n,k)>0$ if $k>\frac{2n}{5}$.\footnote{In fact, for infinitely many odd integers $n$, we can state this as $t(n,k)>0$ if and only if $k>\frac{2n}{5}$.}
It is natural to ask for the exact value of $t(n,k)$ in the range $\frac{2n}{5}<k\leq \frac{n}{2}$.
The following theorem proved by Lo \cite{Lo09} offers a comprehensive understanding of nearly all integers $k$ within this range when the odd integer $n$ is sufficiently large.

\begin{thm}[Lo, \cite{Lo09}]\label{Thm:Lo's thm}
For every odd integer $n\ge 10^7$ and even integer $k$ with $\frac{2n}{5}+\frac{12\sqrt{n}}{5}< k< \frac{n}{2}$,
$$t(n,k)=\frac k 4(3k-n-1).$$
Moreover, the extremal graphs for $t(n,k)$ must belong to the family $\dG(n,k)$ given by Definition~\ref{dfn_fig}.
\end{thm}

Very recently, Cambie, de Joannis de Verclos, and Kang \cite{CdVK} have revisited this problem by investigating the concept of regular Tur\'an numbers.
One of their main results states that $t(n,k)\geq \frac{n^2}{300}$ for all odd $n$ and even $k$ with $k>\frac {2n}{5}.$
They also highlighted the following conjecture (see Conjecture 20 in \cite{CdVK}).

\begin{conjecture}[Cambie, de Verclos and Kang, \cite{CdVK}]\label{Conj: minimum number of trianglees}
For every odd integer $n$ and even integer $k$ with $\frac{2n}{5}< k< \frac{n}{2}$,
$t(n,k)=\frac k 4(3k-n-1).$ Moreover, the extremal graphs must belong to $\dG(n,k)$.
\end{conjecture}

The main contribution of this paper is to establish this conjecture for sufficiently large integers $n$.

\begin{thm}\label{Thm: Main}
For every odd integer $n\ge 10^9$ and even integer $k$ with $\frac{2n}{5}< k< \frac{n}{2}$,
$t(n,k)=\frac k 4(3k-n-1).$ Moreover, the extremal graphs must belong to $\dG(n,k)$.
\end{thm}

Our proof of Theorem~\ref{Thm: Main} employs the approach utilized in \cite{Lo09}, incorporating several innovative technical ideas.
For a brief overview of the proof, we direct readers to the initial portions of both Section~3 and Section~4.
As a direct corollary of Theorem~\ref{Thm: Main}, we can improve the aforementioned lower bound of $t(n,k)$ by Cambie et al. in the following.

\begin{corollary}\label{Cor: Improvement of t(n,k)}
For every odd integer $n\ge 10^9$ and even integer $k>\frac{2n}{5}$,
we have $t(n,k)\ge \frac{(n+1)^2}{50}$, where the equality holds if and only if $k=\frac{2(n+1)}{5}$ and $n+1$ is divisible by $5$.
\end{corollary}

The rest of the paper is organized as follows.
Section 2 provides some preliminaries and introduces the graph family $\dG(n,k)$.
Section 3 focuses on reducing the proof of Theorem~\ref{Thm: Main} to the key intermediate result, namely Theorem~\ref{Thm: Main Reduced}.
The proof of Corollary~\ref{Cor: Improvement of t(n,k)} will also be presented at the end of Section 3.
Section 4 presents the complete proof of Theorem~\ref{Thm: Main Reduced}.
The concluding section offers some related remarks.

\section{Preliminaries}
We use standard notations on graphs throughout the paper.
Let $G$ be a graph.
We use $\Delta(G)$ to denote the maximum degree of $G$.
For a subset $U\subseteq V(G)$, the subgraph of $G$ induced by the vertex set $U$ is denoted by $G[U]$, while the subgraph obtained from $G$ by deleting all vertices in $U$ is expressed by $G\backslash U$.
Let $N_{G}(U)=\bigcap_{v\in U} N_{G}(v)$ be the set consisting of common neighbors of all vertices of $U$ in $G$.
For a subset of edges $A\subseteq E(G)$, we denote $G\setminus A$ the subgraph of $G$ obtained by deleting all edges in $A$.
Suppose that $U,W\subset V(G)$ are disjoint.
We denote $E_{G}(U,W)$ to be the set of edges of $G$ between $U$ and $W$ and let $e_{G}(U,W)=|E_{G}(U,W)|$.
For a vertex $v\in V(G)$ and a subset $X\subseteq V(G),$
we define $d_X(v)=|N_G(v)\cap X|.$
For an integer $k\geq 1$, a {\it $k$-factor} of $G$ is a spanning $k$-regular subgraph of $G$.
We use {\it $T(G)$} to denote the number of triangles in $G$.
For a vertex $v\in V(G)$ and an edge $e\in E(G)$, we use {\it $T(v)$} and {\it $T(e)$} to denote the number of triangles containing the vertex $v$ and the edge $e$ respectively.
For positive integers $m$, we write $[m]$ for the set $\{1,2,...,m\}$.
We often drop above subscripts when they are clear from context, and omit floors and ceilings whenever they are not critical.

\begin{figure}[ht!]
    \centering\includegraphics[height=8cm,width=18cm]{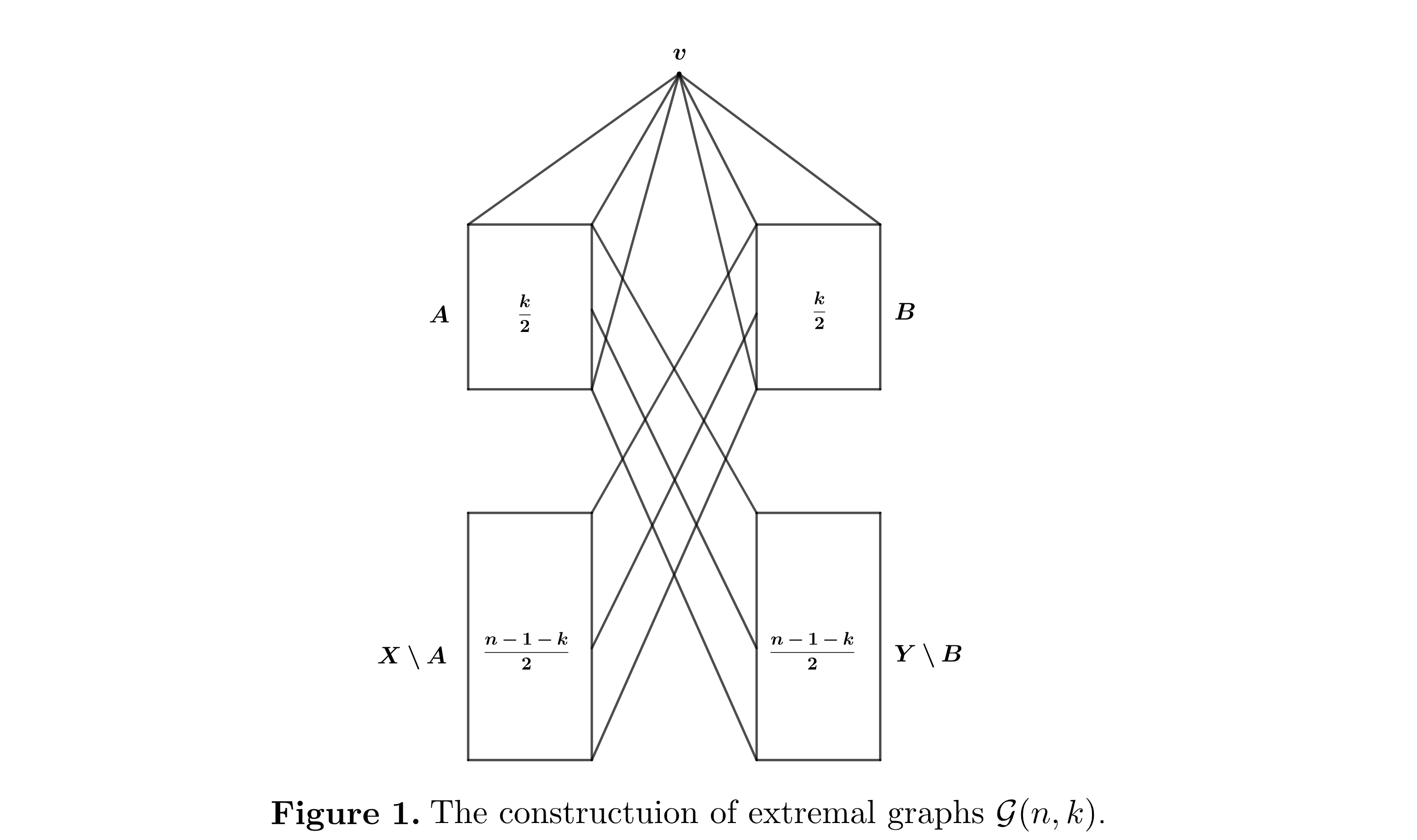}
\end{figure}

Let $n$ be an odd integer $n$ and $k$ be an even integer $k$ satisfying $\frac{2n}{5}< k< \frac{n}{2}$.
We now define an important family of $n$-vertex $k$-regular graphs $\dG(n,k)$ as follows (see Figure~1).
\begin{dfn}\label{dfn_fig}
Let $G$ be a complete bipartite graph with parts $X, Y$, each of size $\frac{n-1}{2}$.
Let $A\subseteq X$, $B\subseteq Y$ be two subsets of size $\frac{k}{2}.$
The family $\dG(n,k)$ consists of all $n$-vertex graphs obtained by the following procedures:
adding a new vertex $v$ to $G$, joining $v$ to all vertices of $A$ and $B$,
removing the edges of a $\frac{n-2k+1}{2}$-factor from the induced subgraph $G[A\cup B]$, and removing the edges of a $\frac{n-2k-1}{2}$-factor from the induced subgraph $G[(X\setminus A)\cup (Y\setminus B)]$.
\end{dfn}
\noindent It is easy to see that any graph $G\in \dG(n,k)$ is $k$-regular.
Moreover, all triangles in $G$ contain the vertex $v$, and every neighbor of $v$ is contained in exactly $\frac k 2-\frac{n-2k+1}{2}=\frac{3k-n-1}{2}$ triangles.
Thus, the number of triangles in any graph $G\in \dG(n,k)$ is equal to $\frac k 4(3k-n-1).$

\section{Proofs of Theorem~\ref{Thm: Main} and Corollary~\ref{Cor: Improvement of t(n,k)}}
In this section, we prove our main result -- Theorem~\ref{Thm: Main}.
We make use of Lo's approach \cite{Lo09}, where he proved that for any odd integer $n\geq 10^7$ and even $k$ satisfying $\frac{2n}{5}+\frac{12\sqrt{n}}{5}< k< \frac{n}{2}$,
let $G$ be an $n$-vertex $k$-regular graph with the minimum number of triangles, then $G\in \dG(n,k)$.
A key intermediate step in Lo's proof is to show, by using Andr\'asfai-Erd\H{o}s-S\'os Theorem \cite{AES}, that such $G$ can be made bipartite by deleting $O(k^{\frac 3 2})$ edges.
In our proof, we present a novel insight wherein, by employing a more thorough structural analysis instead of Andr\'asfai-Erd\H{o}s-S\'os Theorem,
we establish that, when the broader condition $\frac{2n}{5} < k < \frac{n}{2}$ holds, one can also make the graph $G$ bipartite by removing $O(k^{\frac{3}{2}})$ edges.
To be formal, we prove the following.

\begin{thm}\label{Thm: Main Reduced}
Let $G$ be an $n$-vertex $k$-regular graph, where $n\ge 10^9$ is odd and $\frac{2n}{5}< k< \frac{n}{2}$ is even.
If $T(G)\le \frac k 4(3k-n-1)$, then $G$ can be made bipartite by deleting at most $\frac {15}{2} k^{\frac 3 2}$ edges.
\end{thm}

In the remaining of this section, we will demonstrate the proof of Theorem~\ref{Thm: Main} while assuming the validity of Theorem~\ref{Thm: Main Reduced}.
We also need a useful lemma proved by Lo (see Lemma 3.2 in \cite{Lo09}).
For a graph $H$, let $\phi(H)$ be the sum of squares of degrees of all vertices of $H.$

\begin{lem}[Lo \cite{Lo09}]\label{Lem: Lo's lemma}
Let $r>4500$ be a positive integer, and let $H$ be an $n$-vertex graph with $\Delta(H)\le r$ and $e(H)=\beta r$ for $1\le \beta < \frac{6}{5}$.
Then $\phi(H)\le (\beta^2-2\beta+2)r^2 + (5\beta-4)r$.
Furthermore, equality holds if and only if $H$ has a vertex $u$ of degree $r$, a vertex $v\in N(u)$ with degree $(\beta-1)r+1$ and $n-r-1$ isolated vertices.
\end{lem}

To ensure completeness, we provide a comprehensive proof, adapted from \cite{Lo09}, illustrating the reduction of Theorem~\ref{Thm: Main} to Theorem~\ref{Thm: Main Reduced}, as follows.

\medskip

\noindent\textbf{Proof of Theorem~\ref{Thm: Main} (Assuming Theorem~\ref{Thm: Main Reduced}).}
The proof presented here has been adjusted from Lo's work \cite{Lo09}.
Let $G$ be an $n$-vertex $k$-regular graph, where $n\ge 10^9$ is odd and $\frac{2n}{5}< k< \frac{n}{2}$ is even,
such that $T(G)$ is minimum.
By the construction of $\dG(n,k),$
we see that $T(G)\le \frac k 4(3k-n-1)$.

First partition $V(G)=X\cup Y$ with $e(X,Y)$ maximal.
Without loss of generality, suppose that $|X|>|Y|.$
Define $e(X)=\frac 1 2\beta k.$
By Theorem~\ref{Thm: Main Reduced}, $G$ can be made bipartite by deleting at most $\frac{15}{2} k^{\frac 3 2}$ edges.
Since $e(X)+e(Y)$ is minimum, we have that $e(X)\le \frac{15}{2} k^{\frac 3 2}$, which implies that $\beta\le 15\sqrt k.$
Since $2e(X)=k|X|-e(X,Y)\geq k(|X|-|Y|)$, we get that $\beta\ge |X|-|Y|\ge 1.$

For any edge $uv\in E(G[X]),$
$T(uv)\ge d_Y(u)+d_Y(v)-|Y|= 2k-|Y|-(d_X(u)+d_X(v)),$
where the right-hand side only counts the number of triangles $uvy$ with $y\in Y$.
Adding up, we have
\begin{align}\label{Equ: lower bound of T(G)}
T(G)\ge e(X)(2k-|Y|)-\phi(G[X])=\frac 1 2\beta k(2k-|Y|)-\phi(G[X]).
\end{align}
Using $T(G)\le \frac k 4(3k-n-1)$, $k>\frac{2n}{5},$ $|Y|\le \frac{n-1}{2}$ and $\beta\ge 1$,
we can get a lower bound of $\phi(G[X])$ that (see the equation (4.3) in \cite{Lo09})
\begin{align}\label{Equ: lower bound of phi(X)}
      \phi(G[X])\ge \frac 1 8 (3\beta-1) k^2.
\end{align}

In order to apply Lemma~\ref{Lem: Lo's lemma} with $H:=\phi(G[X])$,
we now proceed to show that $\beta<\frac 6 5.$
Let $S$ be the set of vertices $x\in X$ with $d_X(x)\ge \alpha\sqrt{\beta k}$,
where $\alpha\le \frac 1 2\sqrt{k/\beta}$ is some real to be decided later.
Then, $|S|\le 2e(X)/\alpha\sqrt{\beta k}= \sqrt{\beta k}/\alpha$.
Also $\Sigma:=\sum_{u\in S}d_X(u)\le e(X)+\binom{|S|}{2}\le (\frac 1 2+\frac 1 {2\alpha^2})\beta k.$
By the maximality of $e(X,Y)$, we have $d_X(x)\leq \frac{k}{2}$ and $d_Y(y)\le \frac k 2$ for all $x\in X$ and $y\in Y$.
Hence, we have
\begin{align*}
\phi(G[X])&=\sum_{u\in S}d_X(u)^2+\sum_{v\in X\setminus S}d_X(v)^2
\le\frac k 2\cdot\Sigma+\alpha\sqrt{\beta k}\cdot\left(2e(X)-\Sigma\right)\\
&\le \frac 1 2 \left(\frac 1 2+\frac 1 {2\alpha^2}\right)\beta k^2+\alpha\left(\frac 1 2-\frac 1 {2\alpha^2}\right)(\beta k)^{\frac{3}{2}}
\le \left(\frac 1 4+\frac 1 {4\alpha^2}\right)\beta k^2+\frac \alpha 2 (\beta k)^{\frac{3}{2}},
\end{align*}
where the second-to-last inequality holds since $\alpha\le \frac 1 2\sqrt{k/\beta}$ and $\Sigma\leq (\frac 1 2+\frac 1 {2\alpha^2})\beta k$.
Recall $\beta\le 15\sqrt k$ and $k>\frac{2n}{5}\geq 2^{24}$.
So we have $(k/\beta)^{1/6}<\frac 1 2\sqrt{k/\beta}.$
Thus, we can set $\alpha=(k/\beta)^{1/6}$ to get that
$$\phi(G[X])\le \frac{1}{4}\beta k^2+\frac 1 4 \beta^{\frac 4 3}k^{\frac 5 3}+\frac 1 2 \beta^{\frac 4 3}k^{\frac 5 3}=\frac{1}{4}\beta k^2+\frac 3 4 \beta^{\frac 4 3}k^{\frac 5 3}.$$
Combining with the inequality \eqref{Equ: lower bound of phi(X)}, we get that
$$\beta-1-6\beta^{\frac 4 3}k^{-\frac{1}{3}}\le 0.$$
Let $f(\beta)=\beta-1-6\beta^{\frac 4 3}k^{-\frac{1}{3}}.$
To prove $\beta< \frac 6 5$, it is enough to show that $f(\beta)>0$ when $\beta \in[\frac 6 5, 15\sqrt k]$.
Since $f''(\beta)=-\frac 8 3 \beta^{-\frac 2 3}k^{-\frac{1}{3}}<0$,
by convexity, we only need to check the cases when $\beta= \frac 6 5$ and $\beta=15\sqrt k.$
If $\beta= \frac 6 5$, we have
$f(\beta)=\frac 1 5-6\cdot (\frac 6 5)^{\frac{4}{3}}\cdot k^{-1/3}>0,$ as $k>2^{24}$.
If $\beta= 15\sqrt k$, we have
$f(\beta)=15\sqrt k-1-6\cdot 15^{\frac 4 3}\cdot k^{\frac 1 3}>0,$ as $k> 2^{24}$.
This completes the proof of that $\beta< \frac 6 5$.

Note that $\frac 6 5>\beta\ge |X|-|Y|\ge 1$.
So $|X|-|Y|$ is an integer which must be $1$.
By applying Lemma~\ref{Lem: Lo's lemma} with $H:=G[X]$ and $r:=\frac k 2$,
we conclude that $\phi(G[X])\le \frac 1 4(\beta^2-2\beta+2)k^2 + \frac 1 2(5\beta-4)k.$
Then, by~\eqref{Equ: lower bound of T(G)}, we have
\begin{align}\label{Eq:T(G)>k/4(3k-n-1)}
\nonumber
T(G)&\ge \frac 1 2\beta k (2k-|Y|)-\phi(G[X])
\ge \frac 1 2\beta k \left(2k-\frac{n-1}{2}\right)-\frac 1 4(\beta^2-2\beta+2)k^2 - \frac 1 2(5\beta-4)k\\
&=\frac k 4 \Big( 4\beta k-(\beta^2-2\beta+2)k-\beta(n-1)-2(5\beta-4)\Big)
\ge \frac{k}{4}(3k-n-1),
\end{align}
where the last inequality holds because $g(\beta):=4\beta k-(\beta^2-2\beta+2)k-\beta(n-1)-2(5\beta-4)\ge g(1)=3k-n-1$,
which follows from the fact $g'(\beta)=6k-2\beta k-(n-1)-10>0$ (as $\beta<\frac 6 5$, $n<\frac {5k}{2}$ and $k>2^{24}$).
Thus, we get the equality in~\eqref{Eq:T(G)>k/4(3k-n-1)}, from which we can easily drive that $\beta=1$, $|Y|=\frac{n-1}{2}$, $|X|=\frac{n+1}{2}$, and $e(G[Y])=0$.
Moreover, by Lemma~\ref{Lem: Lo's lemma}, $G[X]$ consists of a star centered at $v$ with $\frac k 2$ edges and $\frac{n-k-1}{2}$ isolated vertices.
By~\eqref{Equ: lower bound of T(G)}, any edge $vw$ in $G[X]$ is contained in exactly $2k-|Y|-(d_X(v)+d_X(w))=2k-\frac {n-1}{2}-(\frac k 2+1)=\frac{3k-n-1}{2}$ triangles.
Taking into account all of this information, it becomes evident that $G\in \dG(n,k)$, which completes the proof of Theorem~\ref{Thm: Main}.
\QED
\medskip

To conclude this section, we give the proof of Corollary~\ref{Cor: Improvement of t(n,k)}.

\medskip
\noindent\textbf{Proof of Corollary~\ref{Cor: Improvement of t(n,k)}.}
Let $n\ge 10^9$ be odd and $k>\frac{2n}{5}$ be even.
Note that in fact, $k\ge \frac{2(n+1)}{5}.$
By Theorem~\ref{Thm: Main}, when $\frac{2n}{5}< k< \frac{n}{2}$,
we have $t(n,k)=\frac k 4(3k-n-1)$, which is an increasing function with variable $k$.
Thus in this range, we can get that $t(n,k)\ge \frac{n+1}{10}(3\cdot \frac{2n+2}{5}-n-1)=\frac{(n+1)^2}{50},$ where the equality holds if and only if $k=\frac{2(n+1)}{5}$ and $n+1$ is divisible by $5$.
Now consider $k\ge \frac n 2.$
Since $n$ is odd, we see $k\ge \frac{n+1}{2}.$
Let $G$ be any $n$-vertex $k$-regular graph.
Using the classic Moon-Moser inequality \cite{MM65}, we have
$T(G)\ge \frac{4e(G)}{3}(\frac{e(G)}{n}-\frac{n}{4})$.
Since $e(G)=\frac{kn}{2},$
it gives that $T(G)\ge \frac{2kn}{3}(\frac{k}{2}-\frac{n}{4})> \frac{n^2}{3}(\frac{n+1}{4}-\frac{n}{4})=\frac{n^2}{12}>\frac{(n+1)^2}{50}$, which completes the proof of Corollary~\ref{Cor: Improvement of t(n,k)}.
\QED

\section{Proof of Theorem~\ref{Thm: Main Reduced}}
This section is devoted to the proof of Theorem~\ref{Thm: Main Reduced}.
Throughout this section, we assume that $n\ge 10^9$ is odd and $k$ is even satisfying $\frac{2n}{5}< k< \frac{n}{2}$, and $G$ is an $n$-vertex $k$-regular graph with
\begin{equation}\label{Equ: T(G)<k/4(3k-n-1)}
T(G)\le \frac k 4(3k-n-1).
\end{equation}
Our goal is to show that $G$ can be made bipartite by deleting at most $\frac {15}{2} k^{\frac 3 2}$ edges.
Call an edge $e\in E(G)$ {\it\textbf{heavy}} if $T(e)\ge \frac{3k-n-1}{3}.$
Let $H$ be the spanning subgraph of $G$ whose edge set consists of all heavy edges of $G$.
Let $$U=\{u\in V(H):d_H(u)\ge \frac 3 2 \sqrt{k}\}\ \ \text{and} \ \ G'=G\setminus E(H).$$
Building upon the approach used in \cite{Lo09}, it is crucial to show that $G'$ is ``close'' to a bipartite graph.
To make the first step, we have the following proposition proved in \cite{Lo09}.

\begin{prop}\label{Pro: G' is triangle-free}
$G'$ is triangle-free.
\end{prop}
\begin{proof}
It is enough to show that any triangle in $G$ contains at least one heavy edge.
Let $T$ be any triangle in $G$ with edges $e_1,e_2,e_3$.
For $i\in\{0,1,2,3\}$, let $m_i$ denote the number of vertices in $G$ with exact $i$ neighbors in $T$.
By double-counting, we have $\sum_{i=0}^3 m_i=n,$ and $\sum_{i=0}^3 im_i=3k,$
which implies that $T(e_1)+T(e_2)+T(e_3)=m_2+3m_3\geq m_2+2m_3=3k-n+m_0\ge 3k-n.$
Hence at least one of the edges of $T$ is heavy.
This proves that $G'=G\setminus E(H)$ is triangle-free.
\end{proof}

The next proposition is important in our proof, which says that $G'\setminus U$ does not contain five-cycles.

\begin{prop}\label{Pro: G'-U is C5-free}
$G'\setminus U$ is $C_5$-free.
\end{prop}

The remaining proof is structured into two subsections.
In Subsection \ref{Subsec:reduction}, we demonstrate how the proof of Proposition \ref{Pro: G'-U is C5-free} leads to the derivation of Theorem \ref{Thm: Main Reduced}.
In Subsection \ref{Subsec: Proof of key claim}, we complete the proof of Proposition \ref{Pro: G'-U is C5-free}.

\subsection{Reducing to Proposition~\ref{Pro: G'-U is C5-free}}\label{Subsec:reduction}

\noindent\textbf{Proof of Theorem~\ref{Thm: Main Reduced} (Assuming Proposition~\ref{Pro: G'-U is C5-free}).}
First, by double-counting the number of pairs of edges contained in triangles, and the definition of $H$, we can get that
$$T(G)=\frac 1 3\sum_{e\in E(G)}T(e)\ge \frac 1 3\sum_{e\in E(H)}T(e)\ge \frac 1 3 \cdot \frac{3k-n-1}{3}\cdot|E(H)|.$$
Thus, by \eqref{Equ: T(G)<k/4(3k-n-1)}, we have $|E(H)|\le \frac 9 4 k.$
Then, by the definition of $U$, we have that
$$|U|\le \frac{2|E(H)|}{3\sqrt k/2}\le \frac{9k/2}{3\sqrt k/2}= 3\sqrt k.$$

Next, we claim that $G'\setminus U$ is bipartite.
Suppose not. Then there is an odd cycle $C_{2\ell+1}$ with minimum length $2\ell+1$ in $G'\setminus U.$
By Propositions~\ref{Pro: G' is triangle-free} and \ref{Pro: G'-U is C5-free}, we have $\ell\ge 3.$
Let $V(C_{2\ell+1})=\{v_1,v_2,...v_{2\ell+1}\}.$
By Proposition~\ref{Pro: G' is triangle-free}, we have $N_{G'}(v_1)\cap N_{G'}(v_2)=\emptyset.$
By the minimality of $C_{2\ell+1},$ we can conclude that $N_{G'}(v_1)\cap N_{G'}(v_{\ell+2})\subseteq U$ and $N_{G'}(v_2)\cap N_{G'}(v_{\ell+2})\subseteq U.$
Otherwise, one can find a cycle in $G'\setminus U$ with odd length $t\in\{\ell+2,\ell+3\}$;
since $\ell+3<2\ell+1$ when $\ell\ge3,$ this contradicts our choice of $C_{2\ell+1}$.
For any vertex $v\in V(G')\setminus U$, $d_H(v)\le \frac 3 2 \sqrt k$, which also says that $|N_{G'}(v)|\ge k-\frac 3 2 \sqrt k.$
Since $|U|\le 3\sqrt k$, we can get that
$$n\ge |N_{G'}(v_1)\cup N_{G'}(v_2)\cup N_{G'}(v_{\ell+2})|\ge |N_{G'}(v_1)|+| N_{G'}(v_2)|+|N_{G'}(v_{\ell+2})|-|U|\ge 3(k-\frac 3 2 \sqrt k)-3\sqrt k,$$
which implies that $n\ge 3k-\frac {15}{2}\sqrt k>\frac 5 2 k>n$ (as $k>\frac{2}{5}n\geq 10^8$).
This contradiction proves the claim.

Therefore, $G'$ can be made bipartite by deleting all edges with at least one endpoint in $U$,
which is at most $\binom{|U|}{2}+|U|(n-|U|)\le n|U|-\frac 1 2 |U|^2\le \frac 5 2 k |U|-\frac 1 2|U|^2.$
Since $\frac 5 2 k |U|-\frac 1 2|U|^2$ is monotonically increasing with respect to $|U|$ when $|U|\le 3\sqrt k$,
$G'$ can be made bipartite by deleting at most $\frac 5 2 k |U|-\frac 1 2|U|^2\le \frac {15} 2 k^{\frac 3 2}-\frac 9 2 k$ edges.
Note that $G'=G\setminus E(H)$ and $|E(H)|\le \frac 9 4 k$.
So $G$ can be made bipartite by deleting at most $(\frac {15} 2 k^{\frac 3 2}-\frac 9 2 k)+|E(H)|< \frac {15} 2 k^{\frac 3 2}$ edges,
completing the proof of Theorem~\ref{Thm: Main Reduced}.\QED

\subsection{Proof of Proposition~\ref{Pro: G'-U is C5-free}}\label{Subsec: Proof of key claim}
We begin with an outline of the proof for Proposition~\ref{Pro: G'-U is C5-free}.
Our approach primarily relies on stability-based reasoning.
To begin, we assume, for the sake of contradiction, that $G'\setminus U$ contains a $C_5$.
We then establish that $G'$ closely resembles a balanced blow-up of $C_5$.
Furthermore, by analyzing the addition of edges from $H$ back into $G$, we observe an excess of triangles in $G$.
This contradicts our initial assumption that $T(G)\leq \frac{k}{4}(3k-n-1)$.
Throughout the proof, all indices will be modulo 5.

Suppose on the contrary that there is a $C_5$, say $u_1u_2u_3u_4u_5u_1$, in $G'\setminus U$.
For $i\in [5]$, we define $N^1_{i}=N_{G'}(u_{i-1})\cap N_{G'}(u_{i+1})$.
By Proposition~\ref{Pro: G' is triangle-free}, $G'$ is triangle-free, so $N_{G'}(u_{i})\cap N_{G'}(u_{i+1})=\emptyset$.
Then, for $i\in[5]$, we can get that
$$n\ge |N_{G'}(u_{i-1})\cup N_{G'}(u_{i})\cup N_{G'}(u_{i+1})|\ge d_{G'}(u_{i-1})+d_{G'}(u_i)+d_{G'}(u_{i+1})-|N^1_{i}|,$$
which implies that
\begin{equation}\label{Equ: lower bound of Ni}
|N^1_i|\ge \sum_{j=i-1}^{i+1}d_{G'}(u_j)-n.
\end{equation}
Since $G'$ is triangle-free, for different $i,j\in[5],$ we also have $N^1_{i}\cap N^1_{j}=\emptyset.$
Thus, by \eqref{Equ: lower bound of Ni}, we have
\begin{align}\label{Equ: bound of union of Nis}
 n\ge \left|\bigcup_{i=1}^5 N^1_{i}\right|=\sum_{i=1}^5|N^1_i|\ge 3\sum_{i=1}^5d_{G'}(u_i)-5n,
\end{align}
By \eqref{Equ: lower bound of Ni} again (summing over indices $[5]\backslash \{i\}$), we get
$n-|N^1_i|\geq \Big(3\sum_{i=1}^5d_{G'}(u_i)-\sum_{j=i-1}^{i+1}d_{G'}(u_j)\Big)-4n,$
which implies that
$|N^1_i|\le 5n-3\sum_{j=1}^5d_{G'}(u_j)+\sum_{j=i-1}^{i+1}d_{G'}(u_j).$
Therefore, for any $i\in[5],$ we have
\begin{equation}\label{Equ: bounds of Ni}
\sum_{j=i-1}^{i+1}d_{G'}(u_j)-n\le |N^1_i|\le 5n-3\sum_{j=1}^5d_{G'}(u_j)+\sum_{j=i-1}^{i+1}d_{G'}(u_j)
\end{equation}
Note that $d_{G'}(u)\ge k-\frac{3}{2}\sqrt k$ holds for any $u\in V(G')\setminus U$.
Hence we have $\sum_{j=1}^5d_{G'}(u_j)\ge 5(k-\frac{3}{2}\sqrt k).$

Now we choose $v_1v_2v_3v_4v_5v_1$ to be a $C_5$ in $G'$ such that
$$\sum_{i=1}^5d_{G'}(v_i)=\max \left\{\sum_{j=1}^5d_{G'}(w_i):G'[\{w_1,w_2,w_3,w_4,w_5\}] \hbox{ contains a } C_5\right\}.$$
Clearly, the above cycle $u_1u_2u_3u_4u_5u_1$ is also a $C_5$ in $G'$. So we have
\begin{align}\label{Equ:dG'(vi)>k-3/2sqrt k}
\sum_{i=1}^5d_{G'}(v_i)\ge \sum_{i=1}^5d_{G'}(u_i) \ge 5\left(k-\frac{3}{2}\sqrt{k}\right).
\end{align}
We partition the vertex set $V(G)$ using the information from the cycle $v_1v_2v_3v_4v_5v_1$ in the following (see Figure~2).
For $i\in[5]$, let $a_i:=d_H(v_i)=k-d_{G'}(v_i)$,
$$a:=\sum_{i=1}^5 a_i=5k-\sum_{i=1}^5d_{G'}(v_i), ~~ N_{i}:=N_{G'}(v_{i-1})\cap N_{G'}(v_{i+1})~~ \text{and}~~ Z:=V(G)\setminus (\cup_{i=1}^5N_i).$$
Note that \eqref{Equ: lower bound of Ni}, \eqref{Equ: bound of union of Nis} and \eqref{Equ: bounds of Ni} also hold for $N_i$ and $v_i.$
\begin{figure}[ht!]
\centering\includegraphics[height=8cm,width=15cm]{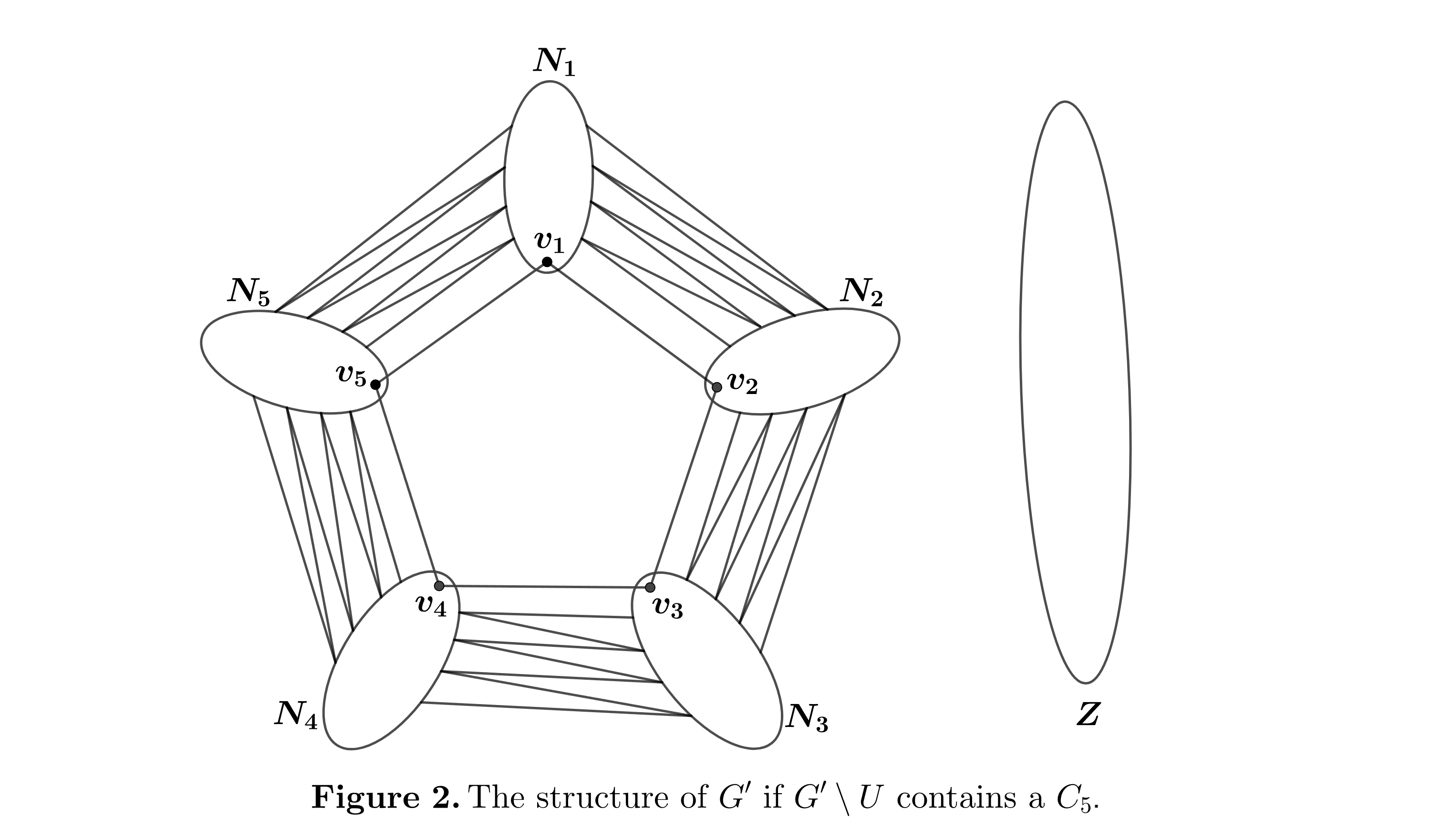}
\end{figure}
Since $\sum_{i=1}^5d_{G'}(v_i)=5k-a$, by~\eqref{Equ: bound of union of Nis}, we have that
\begin{equation}\label{Equ: bound of k}
n\ge \sum_{i=1}^5|N_i|\ge 3(5k-a)-5n = 15k-5n-3a,
\end{equation}
which implies that $\frac{2}{5}n < k \le \frac{2}{5}n+\frac{a}{5}.$
From now on, we define $b:=\frac{5k}{2}-n.$
Since $k$ is even, $b$ is an integer satisfying that $1\le b\le \frac{a}{2}.$
By \eqref{Equ:dG'(vi)>k-3/2sqrt k}, we also see that $2\leq a\leq \frac{15\sqrt{k}}{2}$.

Next, we show that $a$ can be bounded from above by an absolute constant.
We first point out that for any $i\in[5]$ and any vertex $w\in N_i$,
it holds $d_{G'}(w)\le d_{G'}(v_i)$ (equivalently, $d_{H}(w)\ge d_{H}(v_i)=a_i$).
Otherwise, $G'[\{v_1,v_2,v_3,v_4,v_5,w\}-\{v_i\}]$ also contains a $C_5$ with
$\sum_{i=1}^5d_{G'}(v_i)+d_{G'}(w)-d_{G'}(v_i)> \sum_{i=1}^5d_{G'}(v_i)$,
a contradiction to the definition of $v_1v_2v_3v_4v_5v_1$.
Also note that $\sum_{i=1}^5\sum_{w\in N_i}d_{H}(w)\le 2|E(H)|\le \frac 9 2 k$.
Using \eqref{Equ: bounds of Ni} (for $N_i$) and the facts that $n<\frac 5 2 k$ and $a\le \frac{15\sqrt{k}}{2},$
we have that
\begin{align*}
         \frac 9 2 k
         &\ge \sum_{i=1}^5\sum_{w\in N_i}d_{H}(w)\ge \sum_{i=1}^5a_i|N_i|\ge \sum_{i=1}^5a_i\Big(3k-\sum_{j=i-1}^{i+1}a_j-n\Big)\\
         &\ge \frac{k}{2}\sum_{i=1}^5a_i-\sum_{i=1}^5\sum_{j=i-1}^{i+1}a_ia_j\ge
         \frac{ak}{2}-a^2\ge \frac{ak}{2}-\frac{225}{4}k,
\end{align*}
which implies that $a\le 121$.
Therefore in the rest of the proof, it suffices to consider the cases
\begin{equation}\label{Equ:a<121}
2\le a\le 121.
\end{equation}

\subsubsection{Refining the structure}
To deal with the remaining cases, we establish more inequalities and accurate estimations to enhance the graph structure.
Let $i\in [5]$.
First, by substituting $a$ and $b$ and using the fact $3k-a\leq \sum_{j=i-1}^{i+1}d_{G'}(v_j)\leq 3k$,
we can refine the inequality \eqref{Equ: bounds of Ni} and obtain the following new inequality
 	\begin{equation}\label{Equ:new bounds of N_i}
 	\frac{k}{2}-a+b=3k-a-n\le |N_i|\le 5n-3(5k-a)+3k=\frac{k}{2}+3a-5b.
 	\end{equation}

Next, we bound $|Z|$ in terms of $a$ and $b$.
Note that $G'$ is triangle-free.
So any vertex of $Z$ is adjacent to at most one vertex of $\{v_1,v_2,v_3,v_4,v_5\}$ in $G'$.
Thus, $|Z|\ge \sum_{i=1}^5(d_{G'}(v_i)-|N_{i-1}|-|N_{i+1}|)=5k-a-2\sum_{i=1}^5|N_i|=5k-a-2(n-|Z|),$
which implies that $|Z|\le a-2b$.
On the other hand, for any $i\in [5]$, $|N_{i-1}|+|N_{i+1}|\le d_{G'}(v_i)$.
Thus, $\sum_{i=1}^5|N_i|\le (\sum_{i=1}^5d_{G'}(v_i))/2=(5k-a)/2$,
and $|Z|=n-\sum_{i=1}^5|N_i|\ge (\frac{5}{2}k-b)-\frac{5k-a}{2}=\frac{a}{2}-b$.
Summarizing, we have
\begin{equation}\label{Equ: bounds of Z}
\frac{a}{2}-b\le |Z|\le a-2b.
\end{equation}

As we discussed at the beginning of this proof,
we want to demonstrate that when adding the edges of $H$ back into $G$,
there will be many triangles in $G$, thus contradicting \eqref{Equ: T(G)<k/4(3k-n-1)}
For this purpose, it is important to classify the edges $e\in E(H)$ and estimate the value of $T(e)$.
Set $\alpha:=\frac 1 {14}$ and let $$A=\{v\in V(G):d_{H}(v)\le \alpha k\}\ \ \text{and}\ \ B=V(G)\setminus A.$$
Since $|E(H)|\le \frac 9 4 k$, we have that $|B|\le \frac{9k/2}{\alpha k}\le \frac{9}{2\alpha}=63$.
Let $E^0:=E(H)\setminus E(H[Z])$
and it is straightforward to see that
$E^0=E_p\cup E_q\cup E_r\cup E_s,$ where $E_p=\{xy\in E(H): x, y\in N_i \hbox{ for some }i\in[5]\}$,
$E_q=\{xy\in E(H): x\in N_i \hbox{ and } y\in N_{i+1} \hbox{ for some } i\in [5] \}$,
$E_r=\{xy\in E(H): x\in N_i,\  y\in N_j \hbox{ and } |i-j|=2 \hbox{ for some } i,j\in[5] \}$ and
$E_s=\{xy\in E(H): x\in Z \hbox{ and } y\in N_i \hbox{ for some } i\in[5]\}$.

Let $E^0=E^1\cup E^2\cup E^3,$ where $E^1=E^0\cap E(H[A])$, $E^2=E^0\cap E(H[A,B])$ and $E^3=E^0\cap E(H[B])$. 		
The next claim shows that most edges of $E^1$ are contained in many different triangles.

 	\begin{claim}\label{Clm: number of triangles}
        For $\lambda\in\{p,q,r,s\},$ let $E_\lambda^1=E_\lambda\cap E^1$. Then, we have
            \begin{itemize}
 			\item[(1)] for any edge $x_1y_1\in E_p^1$, $|N_{G'}(x_1)\cap N_{G'}(y_1)|\ge \frac 6 7 k-840$;
 			\item[(2)] $E_q^1=\emptyset$;
 			\item[(3)] for any edge $x_3y_3\in E_r^1$, $|N_{G'}(x_3)\cap N_{G'}(y_3)|\ge \frac 5 {14} k -1400$;
 			\item[(4)] for any edge $x_4y_4\in E_s^1,$ $|N_{G'}(x_4)\cap N_{G'}(y_4)|\geq \frac{5}{14}k-1400.$
 		\end{itemize}
 	\end{claim}	
 		
    \begin{proof}
 	For the first statement, let $x_1y_1\in E_p^1$.
    We may assume that $x_1,y_1\in N_1$,
    then the common neighbors of $x_1$ and $y_1$ in $G'$ must belong to $N_2\cup N_5\cup Z$, otherwise there is a triangle in $G'$ containing $v_2$ or $v_5$.
    Thus, by \eqref{Equ:new bounds of N_i}, \eqref{Equ: bounds of Z} and the definition of $A$, we have that
    $|N_{G'}(x_1)\cap N_{G'}(y_1)|\ge d_{G'}(x_1)+d_{G'}(y_1)-|N_2|-|N_5|-|Z|\ge 2(1-\alpha)k-2(\frac k 2+3a-5b)-(a-2b)=(1-2\alpha)k-7a+12b\ge \frac 6 7 k-840,$ the last inequality holds since $a\le 121$ and $b\ge 1.$

        For the second statement, suppose on the contrary that there is an edge $x_2y_2\in E_q^1$.
    If $x_2y_2z_2$ is a triangle in $G$ and $z_2\in \bigcup_{i=1}^5N_i$, then $z_2$ must belong to $N_{H}(x_2)\cup N_H(y_2)$, otherwise there exist $w\in\{x_2,y_2\}$ and some $i\in[5]$ such that $v_iwz_2$ is a triangle in $G'$.
    Thus, by \eqref{Equ: bounds of Z} and the definition of $A$, the number of triangles in $G$ containing $x_2y_2$ is at most
    $d_{H}(x_2)+d_{H}(y_2)+|Z|\le 2\alpha k+a-2b\le \frac k 7+121 < \frac {3k-n-1} 3$, where the last inequality holds since $n=\frac 5 2 k-b$ and $k\ge 10^8,$ which contradicts that $x_2y_2\in E(H)$ is heavy.
    Hence, $E_q^1=\emptyset$.

 	For the third statement, let $x_3y_3\in E_r^1$.
    We may assume that $x_3\in N_1$ and $y_3\in N_3$, then the common neighbors of $x_3$ and $y_3$ in $G'$ must belong to $N_2\cup Z$.
    Note $N_{G'}(x_3)\subseteq N_2\cup N_5\cup Z,$ which implies that $N_{G'}(x_3)\setminus N_5\subseteq N_2\cup Z.$
    Similarly, $N_{G'}(y_3)\setminus N_4\subseteq N_2\cup Z.$
    Thus, by \eqref{Equ:new bounds of N_i}, \eqref{Equ: bounds of Z} and the definition of $A$, we have that
    $|N_{G'}(x_3)\cap N_{G'}(y_3)|\ge (d_{G'}(x_3)-|N_5|)+(d_{G'}(y_3)-|N_4|)-|N_2|-|Z|\ge 2(1-\alpha) k-3(\frac k 2+3a-5b)-(a-2b)=\frac k 2-2\alpha k-10a+17b\ge \frac 5 {14} k -1400$.
 		
 	For the last statement, let $x_4y_4\in E_s^1$ with $x_4\in Z$.
      We first claim that there is an integer $i\in [5]$ such that $|N_{G'}(x_4)\cap N_i|, |N_{G'}(x_4)\cap N_{i+2}|\geq \frac 3 7 k-410$ and $|N_{G'}(x_4)\cap N_j|\le 27$ for $j\in[5]\setminus\{i,i+2\}$.	
      Indeed, for any vertex $z\in Z$ and any integer $i\in [5]$,
      since $G'$ is triangle free, there is no edge between $N_{G'}(z)\cap N_i$ and $N_{G'}(z)\cap N_{i+1}$ in $G'$.
      Thus, any vertex $v\in N_{G'}(z)\cap N_{i}$ has degree at most
      $(|N_{i-1}|-|N_{G'}(z)\cap N_{i-1}|)+(|N_{i+1}|-|N_{G'}(z)\cap N_{i+1}|)+|Z|$ in $G'$, which is equivalent to
      $d_H(v)\ge k-(|N_{i-1}|+|N_{i+1}|+|Z|-|N_{G'}(z)\cap N_{i+1}|-|N_{G'}(z)\cap N_{i-1}|).$
      Therefore, we can conclude that for any $z\in Z,$
 	\begin{equation}\label{Equ: bound of neighborhood of z in Ni}
 		|N_{G'}(z)\cap N_i|\le \frac{2|E(H)|}{k-(|N_{i-1}|+|N_{i+1}|+|Z|-|N_{G'}(z)\cap N_{i+1}|-|N_{G'}(z)\cap N_{i-1}|)}.
 	\end{equation}
     Since $d_{G'}(x_4)\ge k-\alpha k$, without loss of generality, we may assume that $|N_{G'}(x_4)\cap N_1|\ge (k-\alpha k-|Z|)/5\ge (k-\alpha k-a+2b)/5$.
     By \eqref{Equ:new bounds of N_i} and \eqref{Equ: bounds of Z},
     since $k\ge 10^8$ and $a\le 121$,
     we have $$|N_{G'}(x_4)\cap N_2|\le \frac{9k/2}{k-2(\frac k 2 +3a-5b)-(a-2b)+(k-\alpha k-a+2b)/5}=\frac{45k}{\frac {13} 7 k-72a+124b}\le 27.$$
     Similarly, we have $|N_{G'}(x_4)\cap N_5|\le 27$.
     Furthermore, without loss of generality, we may assume that
     $|N_{G'}(x_4)\cap N_3|\ge (k-\alpha k-|N_1|-|Z|-54)/2\ge (\frac k 2-\alpha k-4a+7b-54)/2$.
     Then, by using \eqref{Equ: bound of neighborhood of z in Ni} again, we have $|N_{G'}(x_4)\cap N_4|\le 27$.
 	Therefore, by \eqref{Equ:new bounds of N_i} and \eqref{Equ: bounds of Z},  since $a\le 121$ and $b\ge 1$,
    we can get that
    $$|N_{G'}(x_4)\cap N_1|\ge k-\alpha k -|N_3|-|Z|-81\ge \frac k 2-\alpha k- 4a+7b-81\ge \frac 3 7 k-558.$$
    Similarly, we can also get that $|N_{G'}(x_4)\cap N_3|\ge\frac 3 7 k-558$.
    Thus, we have proved that most of the vertices of $N_{G'}(x_4)$ are contained in $N_1$ and $N_3.$

        Next, we show that $x_4$ and $y_4$ have many common neighbors in $G'.$
        First, we claim that $y_4\notin (N_1\cup N_3)$.
        By symmetry, we only need to show that $y_4\notin N_1.$
        Suppose on the contrary that $y_4\in N_1$, then the third vertex of any triangle in $G$ containing $x_4y_4$ must belong to $(N_{G}(x_4)\cap (N_2\cup N_5\cup Z))\cup (N_{G}(y_4)\cap (N_1\cup N_3\cup N_4))$.
        It is easy to see that $|N_{G}(x_4)\cap (N_2\cup N_5\cup Z)|\le d_{H}(x_4)+|N_{G'}(x_4)\cap N_2|+|N_{G'}(x_4)\cap N_5|+|Z|\le \alpha k +a-2b+54$
        and $|N_{G}(y_4)\cap (N_1\cup N_3\cup N_4)|\le d_{H}(y_4)\le \alpha k$.
        Thus the number of triangles containing edge $x_4y_4$ is at most $2\alpha k +a-2b+54\le \frac k 7 +175<\frac {3k-n-1} 3$, where the last inequality holds since $n=\frac 5 2 k-b$ and $k\ge 10^8,$ which contradicts that $x_4y_4\in E(H)$ is heavy.
        If $y_4\in N_2$, then $N_{G'}(y_4)\subseteq N_1\cup N_3\cup Z$.
    By \eqref{Equ:new bounds of N_i} and \eqref{Equ: bounds of Z},
    \begin{align*}
        |N_{G'}(x_4)\cap N_{G'}(y_4)|&\ge |N_{G'}(x_4)\cap (N_1\cup N_3\cup Z)|+d_{G'}(y_4)-|N_1|-|N_3|-|Z|\\
        &\ge 2(\frac 3 7 k-558)+\frac {13}{14}k-2(\frac k 2+3a-5b)-(a-2b)\ge \frac {11}{14}k-1951\ge \frac{5}{14}k-1400.
    \end{align*}
 	If $y_4\in N_4$, then $N_{G'}(y_4)\subseteq N_3\cup N_5\cup Z$.
    By \eqref{Equ:new bounds of N_i} and \eqref{Equ: bounds of Z},
    \begin{align*}
        |N_{G'}(x_4)\cap N_{G'}(y_4)|&\ge |N_{G'}(x_4)\cap (N_3\cup N_5\cup Z)|+d_{G'}(y_4)-|N_3|-|N_5|-|Z|\\
        &\ge \frac 3 7 k-558+\frac {13}{14}k-2(\frac k 2+3a-5b)-(a-2b)\ge \frac{5}{14}k-1400.
    \end{align*}
    Similarly, if $y_4\in N_5$, then $N_{G'}(y_4)\subseteq N_1\cup N_4\cup Z$ and we can show that $|N_{G'}(x_4)\cap N_{G'}(y_4)|\ge \frac{5}{14}k-1400$. This completes the proof of Claim~\ref{Clm: number of triangles}.	
     \end{proof}

\subsubsection{Reducing to the cases $a\in \{2,3\}$}
Having the refined structure, we now proceed to complete the proof of Proposition~\ref{Pro: G'-U is C5-free}.
By double-counting the edges of $E(H[B])\setminus E(H[Z]),$ we have that
$|E^2|+2|E^3|=\sum_{i=1}^5\sum_{w\in B\cap N_i}d_{H}(w)+\sum_{z\in Z\cap B}|N_{H}(z)\cap (\bigcup_{i=1}^5N_i)|.$
Since $E^0=E(H)\setminus E(H[Z]),$
we have that
\begin{equation}\label{Equ:2|E^0|}
 		\begin{split}
 		2(|E^1|+|E^2|&+|E^3|)=2|E^0|=\sum_{i=1}^{5}\sum_{w\in N_i}d_{H}(w)+\sum_{z\in Z}\Big|N_{H}(z)\cap (\bigcup_{i=1}^5N_i)\Big|=\sum_{i=1}^5\sum_{w\in A\cap N_i}d_{H}(w)\\
 		&+\sum_{i=1}^5\sum_{w\in B\cap N_i}d_{H}(w)+\sum_{z\in Z}\Big|N_{H}(z)\cap (\bigcup_{i=1}^5N_i)\Big|\ge \sum_{i=1}^5\sum_{w\in A\cap N_i}d_{H}(w)+|E^2|+2|E^3|.
\end{split}
\end{equation}
Recall that any vertex $w\in N_i$ satisfies that $d_{H}(w)\ge d_{H}(v_i)=a_i$.
By \eqref{Equ:new bounds of N_i}, we have $|N_i|\ge \frac k 2 -a.$
Combining with \eqref{Equ:2|E^0|} and $|B|\le 63,$ we have
\begin{equation}\label{Equ: bound of sum of E^1 and E^2}
 	2|E^1|+|E^2|\ge \sum_{i=1}^5\sum_{w\in A\cap N_i}d_{H}(w)\ge \sum_{i=1}^5a_i|N_i\cap A|\ge \sum_{i=1}^5a_i(|N_i|-|B|)\ge \frac{ak}{2}-a^2-63a.
\end{equation}
Let $T_2$ be a set of triangles containing at least one edge in $E^2$.
Let $M$ be the number of pairs $(e,T)$ satisfying that $T\in T_2$ contains $e\in E^2$.
By the definition of $H$, we get $M\ge \frac{3k-n-1}{3}|E^2|.$
Since any triangle in $G$ contains at most two edges in $E(H[A,B])$, we get $M\le 2|T_2|.$
Hence,
\begin{equation}\label{Equ:T2}
|T_2|\ge \frac{M}{2}\ge \frac{3k-n-1}{6}|E^2|\ge \frac{k}{12}|E^2|.
\end{equation}
Since $G'=G\setminus E(H)$ is triangle-free, all the triangles we have considered in Claim~\ref{Clm: number of triangles} are disjoint from each other and disjoint from the triangles in $T_2.$
Then, using Claim~\ref{Clm: number of triangles} (note that $|E_q^1|=0$), \eqref{Equ: bound of sum of E^1 and E^2} and \eqref{Equ:T2},
we can derive that
	\begin{align}\label{Equ: lower bound of number of triangles for a>4}
            \nonumber
			T(G)
            &\ge \left(\frac 6 7 k-840\right)|E_p^1|+\left(\frac {5}{14} k-1400\right)|E_r^1|+\left(\frac {5}{14} k-1400\right)|E_s^1|+|T_2|\\
            &\ge\frac {k}{6}\Big(|E_p^1|+|E_r^1|+|E_s^1|\Big)+ \frac{k}{12}|E^2|
            =\frac{k}{12}\Big(2|E^1|+|E^2|\Big)
			\ge \frac{a}{24}k^2-\frac{a^2+63a}{12}k.
	\end{align}

Suppose that $a\geq 4$.
Then by \eqref{Equ:a<121}, we have $4\le a\le 121$.
Since $k\ge 10^8$, $n=\frac{5}{2}k-b$ and $b\le \frac{a}{2}$,
it is easy to see that using \eqref{Equ: lower bound of number of triangles for a>4},
we have $T(G)\ge \frac 1 6 k^2-\frac{67}{3}k>\frac{1}{8}k^2+\frac{b-1}{4}k=\frac k 4(3k-n-1),$
a contradiction to our assumption \eqref{Equ: T(G)<k/4(3k-n-1)}.
Hence, we only need to examine the remaining two cases of $a=2$ and $a=3$, which will be addressed in the subsequent subsections.

\subsubsection{The case when $a=3$}
In this case, since $1\le b\le \frac a 2$, we get that $b=1$ and thus $n=\frac 5 2 k-1.$
Then, by~\eqref{Equ: bounds of Z}, $\frac 1 2\le |Z|\le 1,$ which implies that $|Z|=1$.
Let $Z=\{z\}.$
        By \eqref{Equ:new bounds of N_i}, we have that $\frac k 2-2\le |N_i|\le \frac k  2+4$ for $i\in [5]$.
        We first claim that $|B|\ge 2.$
        Indeed, if $|B|\le 1$, then $|E^2|\le k$.
        By \eqref{Equ: bound of sum of E^1 and E^2},
        we have $2|E^1|+|E^2|\ge \frac{ak}{2}-a^2-63a\ge\frac {3}{2}k-200,$
        which implies that $|E^1|\ge \frac{k}{4}-100.$
        Then, since $k\ge 10^8,$ Claim~\ref{Clm: number of triangles} implies that
		\begin{align*}\label{Equ: lower bound of number of triangles when a=3}
		T(G)
        &\ge \left(\frac 6 7 k-840\right)|E_p^1|+\left(\frac {5}{14} k-1400\right)|E_r^1|+\left(\frac {5}{14} k-1400\right)|E_s^1|+\frac{k}{12}|E^2|\\
		&\ge \left(\frac{5}{14}k-1400\right)|E^1|+\frac{k}{12}|E^2|
         \ge \left(\frac{5}{14}k-1400\right)\left(\frac{k}{4}-100\right)+\frac{k^2}{12}
		>\frac{1}{8}k^2=\frac{k(3k-n-1)}{4},
		\end{align*}
		which is a contradiction to our assumption \eqref{Equ: T(G)<k/4(3k-n-1)}.
		Thus, we may assume that $|B|\ge 2$.
        Then, since $|Z|=1$, we can get that $|B\cap N_i|\ge 1$ for some integer $i\in [5]$.
        Without loss of generality, we may assume that there is a vertex $w_1\in B\cap N_1$.

        Next, we show that $|N_{H}(w_1)\cap (N_1\cup N_3\cup N_4)|\ge \frac{k}{28}.$
        Indeed, if $|N_{H}(w_1)\cap (N_1\cup N_3\cup N_4)|< \frac{k}{28}$, then $|N_{H}(w_1)\cap (N_2\cup N_5)|\ge d_{H}(w_1)-\frac{k}{28}-1\ge \frac{k}{28}-1>|B|$.
		We may choose a vertex $w_2\in N_{H}(w_1)\cap N_2\cap A$.
        Then, since $k\ge 10^8,$ the number of triangles in $G$ containing $w_1w_2$ is at most
        $$|N_{H}(w_1)\cap (N_1\cup N_3\cup N_4)|+d_{H}(w_2)+|Z|<\frac{k}{28}+\frac{k}{14}+1< \frac{k}{6}=\frac{3k-n-1}{3},$$ which is a contradiction to $w_1w_2\in E(H)$ is heavy.

        We denote the set of vertex in $(N_2\cup N_5)\cap A$ that are not adjacent to $w_1$ in $G$ by $S.$
        In the following, we prove that $|S|$ is big and the degree of most vertices of $S$ in $H$ is larger than $v_2$ or $v_5.$
        In this way, we can improve the bound of~\eqref{Equ: bound of sum of E^1 and E^2} and get a contradiction.

		Since $|N_{H}(w_1)\cap (N_1\cup N_3\cup N_4)|\ge \frac {k}{28}$,
        $|N_G(w_1)\cap (N_2\cup N_5)|=|N_{G'}(w_1)\cap (N_2\cup N_5)|+|N_H(w_1)\cap (N_2\cup N_5)|\le k-|N_{H}(w_1)\cap (N_1\cup N_3\cup N_4)|\le k-\frac {k}{28}$.
		Thus, we get that
        \begin{equation}\label{Equ:lower bound of S}
         |S|\ge|N_2|+|N_5|-|N_G(w_1)\cap (N_2\cup N_5)|-|B|\ge 2(\frac k 2-2)-k+\frac {k}{28}-63\ge \frac {k}{28}-67.
        \end{equation}
        Next, we analyze the connection between these vertices and $z$.
        For $i\in[5],$ if all the vertices $v_i$ are not contained in $N_{G'}(z)$,
        then similar to \eqref{Equ: lower bound of Ni},
        we can get $|N_i|\ge d_{G'}(v_{i-1})+d_{G'}(v_i)+d_{G'}(v_{i+1})-n+1,$
        which implies that
        $n-1=\sum_{i=1}^5|N_i|\ge 3\sum_{i=1}^5d_{G'}(v_i)-5n+5=3(5k-3)-5n+5,$ and thus we have
        $n\ge \frac 5 2 k-\frac 1 2,$ which is a contradiction to $n=\frac 5 2 k-1.$
        Thus, we may assume that $zv_{j}\in E(G')$ for some $j\in[5].$
        According to our definition of $N_i,$ for $i\neq j$, $zv_i\notin E(G').$
        Then, similar to \eqref{Equ: bound of union of Nis},
        we can get that $n-1=\sum_{i=1}^5|N_i|\ge 3\sum_{i=1}^5d_{G'}(v_i)-5n+2=3(5k-3)-5n+2,$ and thus we have
        $n\ge \frac 5 2 k-1,$
        which implies that inequality holds.
        Therefore, we can get
        $N_{G'}(v_j)=N_{j-1}\cup N_{j+1}\cup \{z\}$
        and for any $w\in N_j,$ $N_{G'}(w)\subseteq N_{j-1}\cup N_{j+1}\cup \{z\}.$
        For $i\in [5],$ $i\neq j,$
        $N_{G'}(v_i)=N_{i-1}\cup N_{i+1}$
        and for any $w\in N_i,$ $N_{G'}(w)\subseteq N_{i-1}\cup N_{i+1}\cup \{z\}.$

        First note that $N_{G'}(z)\cap (N_{j-1}\cup N_{j+1})=\emptyset,$ or we will get a triangle in $G'.$
        Let $Z_j=N_j\setminus (N_{G'}(z)\cup B),$ $Z_{j-2}=(N_{G'}(z)\cap N_{j-2})\setminus B$ and $Z_{j+2}=(N_{G'}(z)\cap N_{j+2})\setminus B.$
        Now, we claim that $|Z_j|,|Z_{j-2}|, |Z_{j+2}| \le \frac{k}{112}.$
        Indeed, if $|Z_j|>\frac{k}{112},$
        then, since $zv_j\in E(G'),$ comparing the vertex of $Z_j$ with $v_j,$
        we can get
        $\sum_{i=1}^5\sum_{w\in A\cap N_i}d_{G'}(w)\le \sum_{i=1}^5(k-a_i)|N_i\cap A|-\frac{k}{112}.$
        By~\eqref{Equ: bound of sum of E^1 and E^2}, this implies that
        $$2|E^1|+|E^2|\ge \sum_{i=1}^5\sum_{w\in A\cap N_i}d_{H}(w)\ge \sum_{i=1}^5a_i|N_i\cap A|+\frac{k}{112}\ge 3\Big(\frac k 2 -65\Big)+\frac{k}{112}=\frac{169}{112}k-200>\frac{3}{2}k,$$
        and by~\eqref{Equ: lower bound of number of triangles for a>4},
        since $k\ge 10^8,$
        we can get the number of triangles in $G$ to be at least,
        $$T(G)\ge \frac{k}{12}\Big(2|E^1|+|E^2|\Big)\ge \frac{k}{12}\cdot\frac{3}{2}k=\frac{1}{8}k^2=\frac{k(3k-n-1)}{4},$$
		which is a contradiction to our assumption \eqref{Equ: T(G)<k/4(3k-n-1)}
        shows that $|Z_j|\le \frac{k}{112}.$
        We denote the property that $2|E^1|+|E^2|>\frac{3}{2}k$ by $(\star).$
        By the analysis above, when we get $(\star),$ we are done.
        The proof of the remaining two cases is similar to the above.
        If $|Z_{j-2}|> \frac{k}{112}$ and $ |Z_{j+2}|=0,$
        then we can get that
        $\sum_{i=1}^5\sum_{w\in A\cap N_i}d_{G'}(w)\le \sum_{i=1}^5(k-a_i)|N_i\cap A|+|Z_{j-2}|-2|Z_{j-2}|\le \sum_{i=1}^5(k-a_i)|N_i\cap A|-\frac{k}{112}.$
        If $|Z_{j-2}|> \frac{k}{112}$ and $ |Z_{j+2}|\neq 0,$
        then we can get that
        $\sum_{i=1}^5\sum_{w\in A\cap N_i}d_{G'}(w)\le \sum_{i=1}^5(k-a_i)|N_i\cap A|+|Z_{j-2}|+|Z_{j+2}|-2|Z_{j-2}||Z_{j+2}|\le \sum_{i=1}^5(k-a_i)|N_i\cap A|-\frac{k}{112}+1.$
        Doing the above process again, we will get a contradiction.
        Therefore, we have $|Z_{j-2}|,|Z_{j+2}|\le \frac{k}{112}.$

        Now, recall that $w_1\in B\cap N_1$,
        $S$ is the set of vertex in $(N_2\cup N_5)\cap A$ that are not adjacent to $w_1$ in $G$
        and $zv_{j}\in E(G')$ for some $j\in[5].$
        By symmetry, it suffices to consider $j\in \{1,2,3\}.$
        If $j=1,$ since all the vertices of $S$ are not adjacent to $w_1$ in $G$,
        comparing these vertices with $v_2$ or $v_5,$ by~\eqref{Equ:lower bound of S},
        we can get that
        $$2|E^1|+|E^2|\ge \sum_{i=1}^5\sum_{w\in A\cap N_i}d_{H}(w)\ge \sum_{i=1}^5a_i|N_i\cap A|+|S|> 3\Big(\frac k 2 -65\Big)+\frac{k}{112}>\frac{3}{2}k,$$
        again we get $(\star)$.
        If $j=2,$ let $S'=S\setminus (Z_2\cup Z_5).$
        Since $|Z_2|,|Z_5|\le \frac{k}{112},$
        we have $|S'|\ge|S|-2\cdot\frac{k}{112}\ge \frac{k}{56}-67.$
        And for $v\in S'\cap N_2,$ $v$ is not adjacent to $z$.
        For $v\in S'\cap N_5,$ $v$ is not adjacent to $z$ and $w_1$.
        Comparing these vertices with $v_2$ or $v_5,$
        we can get that
        $$2|E^1|+|E^2|\ge \sum_{i=1}^5\sum_{w\in A\cap N_i}d_{H}(w)\ge \sum_{i=1}^5a_i|N_i\cap A|+|S'|> 3\Big(\frac k 2 -65\Big)+\frac{k}{112}>\frac{3}{2}k,$$
        again we get $(\star)$.
        If $j=3,$ since $G'$ is triangle-free and $N_{G'}(v_3)=N_2\cup N_4\cup\{z\},$
        $N_{G'}(z)\cap N_2=N_{G'}(z)\cap N_4=\emptyset.$
        Let $S''=S\setminus Z_{5},$
        since $|Z_5|\le \frac{k}{112},$ we have
        $|S''|\ge|S|-\frac{k}{112}\ge \frac{3}{112}k-67.$
        And for any vertex $v\in S''$, $v$ is not adjacent to $z$ and $w_1$.
        Comparing these vertices with $v_2$ or $v_5,$
        we can get that
        $$2|E^1|+|E^2|\ge \sum_{i=1}^5\sum_{w\in A\cap N_i}d_{H}(w)\ge \sum_{i=1}^5a_i|N_i\cap A|+|S''|> 3\Big(\frac k 2 -65\Big)+\frac{k}{112}>\frac{3}{2}k,$$
        again we get $(\star)$.
        Therefore, we have completed the proof when $a=3.$

\subsubsection{The case when $a=2$}
In this case, since $1\le b\le \frac a 2$, we get that $b=1$ and thus $n=\frac 5 2 k-1.$
       By \eqref{Equ: bounds of Z}, we have $|Z|=0$.
       Then, $V(G)=\bigcup_{i=1}^5N_i$.
       By \eqref{Equ:new bounds of N_i}, we have that $\frac k 2-1\le |N_i|\le \frac k 2+1$ for $i\in[5]$.
       Note that the equality holds in \eqref{Equ: bound of union of Nis}, which implies that $N_{G'}(v_i)=N_{i-1}\cup N_{i+1}$ holds for $i\in [5].$
       Moreover, for any $w\in N_i,$ $N_{G'}(w)\subseteq N_{i-1}\cup N_{i+1}.$
       Let $\mathcal{E}=\cup_{i=1}^5 E(H[N_i,N_{i+1}]).$
       We define $H_1=H\setminus \mathcal{E}$
       and $G_1=G'\cup\mathcal{E}.$
       By our definition, $G_1$ is also triangle-free.
       Similarly, we classify the edges of $H_1$ (in fact, here we can think of $H_1$ and $G_1$ as the previous $H$ and $G'$ respectively).
       Let $A_1=\{v\in V(G):d_{H_1}(v)\le \frac k {14}\}$ and $B_1=V(G)-A_1$.
       Since $|E(H_1)|\le|E(H)|\le \frac 9 4 k$, we have $|B_1|\le 63$.

	  Let $F^0=F_p\cup F_r,$ where $F_p=\{xy\in E(H_1): x, y\in N_i \hbox{ for some }i\in[5]\}$
        and
        $F_r=\{xy\in E(H_1): x\in N_i,\  y\in N_j \hbox{ and } |i-j|=2 \hbox{ for some } i,j\in[5] \}.$
        Let $F^1=F^0\cap E(H_1[A_1]),$
        $F^2=F^0\cap E(H_1[A_1,B_1])$ and $F^3=F^0\cap  E(H_1[B_1])$.
		Note that $F^0=E(H_1)$. Thus, we have
		\begin{equation*}
		\begin{split}
		2(|F^1|+|F^2|+|F^3|)&=2|F^0|
		=\sum_{i=1}^{5}\sum_{w\in N_i}d_{H_1}(w)
	    =\sum_{i=1}^5\left(\sum_{w\in A_1\cap N_i}d_{H_1}(w)+\sum_{w\in B_1\cap N_i}d_{H_1}(w)\right),
		\end{split}
		\end{equation*}
		and
		$|F^2|+2|F^3|=\sum_{i=1}^5\sum_{w\in B_1\cap N_i}d_{H_1}(w),$
		which implies that
		\begin{equation}\label{Equ1: bound of E1 for case a=2}
		2|F^1|+|F^2|=\sum_{i=1}^5\sum_{w\in A_1\cap N_i}d_{H_1}(w).
		\end{equation}

		For any integer $i\in [5]$ and any vertex of $w\in N_i\cap B_1$,
        since $|N_{i-1}|+|N_{i+1}|=k-a_i$,
        there are at least $|N_{i-1}|+|N_{i+1}|-(k-d_{H_1}(w))-|B_1|=d_{H_1}(w)-a_i-|B_1|$ vertices of $(N_{i-1}\cup N_{i+1})\cap A_1$ are not adjacent to $w$ in $G$.
        Thus, for $i\in [5]$, summing up the number of vertices of $(N_{i-1}\cup N_{i+1})\cap A_1$ that are not adjacent to $w$ with $w\in N_i\cap B_1$, we can get that
		\begin{equation}\label{Euq2: bound of E1 for case a=2}
		\sum_{i=1}^5\sum_{w\in A_1\cap N_i}(d_{H_1}(w)-a_i)\geq \sum_{i=1}^5\sum_{w\in B_1\cap N_i}(d_{H_1}(w)-a_i-|B_1|).
		\end{equation}
		By \eqref{Equ1: bound of E1 for case a=2} and \eqref{Euq2: bound of E1 for case a=2},
        since $|N_i|\ge \frac k 2-1$ and $|B_1|\le 63,$ we can get that
		\begin{align*}
		2|F^1|
        &=\sum_{i=1}^5\sum_{w\in A_1\cap N_i}d_{H_1}(w)-|F^2|
		\ge \sum_{i=1}^5 a_i|N_i\cap A_1|+ \sum_{i=1}^5\sum_{w\in A_1\cap N_i}(d_{H_1}(w)-a_i)-\sum_{w\in B_1}d_{H_1}(w)\\
		&\ge 2\Big(\frac k 2-1-|B_1|\Big)- \sum_{i=1}^5\sum_{w\in B_1\cap N_i}(a_i+|B_1|)
		\ge k-128-|B_1|(2+|B_1|)
        \ge k-4400.
		\end{align*}
        Let $F_{p}^1=F_p\cap F^1$ and $F_{r}^1=F_r\cap F^1.$
        We can obtain the same conclusion as for the Claim~\ref{Clm: number of triangles} (replace $G'$ and $H$ with $G_1$ and $H_1$ respectively).
        Since $G_1$ is triangle-free and $E(G_1)=(E(G)\setminus E(H)) \cup \mathcal{E},$
        all the triangles we considered are disjoint from each other.
		Using Claim~\ref{Clm: number of triangles} and $k>10^8$, we have
		\begin{equation*}
		T(G)\ge \left(\frac {5}{14}k-1400\right)\left(\frac{k}{2}-2200\right)>\frac{k^2}{8}=\frac{k(3k-n-1)}{4},
		\end{equation*}
        which is a contradiction to our assumption \eqref{Equ: T(G)<k/4(3k-n-1)}.
        We have completed the proof of Proposition~\ref{Pro: G'-U is C5-free}. \QED

\section{Concluding remarks}
Our main result establishes a tight lower bound for the number of triangles in $n$-vertex $k$-regular graphs under certain conditions
(specifically, when $n$ is odd and sufficiently large, and $k$ is an even integer satisfying $\frac{2n}{5}<k<\frac{n}{2}$).
The same problem was also investigated by Lo in \cite{Lo10,Lo12} for the case of $k\geq \frac{n}{2}$.
For related problems concerning the count of cliques $K_{r}$ in regular graphs or graphs with a bounded minimum degree, we recommend referring to \cite{Lo12}.

It is also interesting to explore the minimum number $t_H(n,k)$ of copies of a general graph $H$ in an $n$-vertex $k$-regular graph.
Results of Cambie, de Joannis de Verclos, and Kang \cite{CdVK} indicate that there is a distinction based on whether the chromatic number $\chi(H)=3$ or not (see Theorems 2, 13, and 14 in \cite{CdVK}).
For all odd cycles $C_{2\ell-1}$ where $\ell \geq 2$, Cambie et al. \cite{CdVK} obtained the following tight result:
$$\text{If $n$ is even and $k>\frac{n}{2}$ or $n$ is odd and $k>\frac{2n}{2\ell+1}$, then } t_{C_{2\ell-1}}(n,k)>0.$$
Assuming $n$ is odd and $k$ is even, satisfying $\frac{2n}{2\ell+1}<k<\frac{n}{2}$, it would be of great interest to determine the exact value of $t_{C_{2\ell-1}}(n,k)$.
This endeavor to generalize Theorem~\ref{Thm: Main} appears to be challenging.

	\bibliographystyle{unsrt}


%
%
%
%
%
%
%

\end{document}